\documentclass[11pt,reqno]{amsart}
 
\usepackage{amssymb,dirtytalk,comment,tikz-cd,tikz,color,cite,enumerate,stmaryrd,graphicx}
\usepackage[unicode=true]
 {hyperref}
\hypersetup{
colorlinks=true,
urlcolor=black,
citecolor=blue,
linkcolor=blue,
}

\usepackage{setspace}
\newtheorem{theorem}{Theorem}[section]
\newtheorem{lem}[theorem]{Lemma}
\newtheorem{prop}[theorem]{Proposition}

\theoremstyle{definition}
\newtheorem{definition}[theorem]{Definition}
\newtheorem{example}[theorem]{Example}

\theoremstyle{remark}

\numberwithin{equation}{section}

\newcommand{\Ann}{\mathrm{Ann}}
\newcommand{\Aut}{\mathrm{Aut}}

\allowdisplaybreaks

\begin{document}

\title{On $n$-isoclinism of skew braces}

\author{Risa Arai}
\address{Department of Mathematics, Ochanomizu University, 2-1-1 Otsuka, Bunkyo-ku, Tokyo, Japan}
\email{g2440601@edu.cc.ocha.ac.jp}

\author{Cindy (Sin Yi) Tsang}
\address{Department of Mathematics, Ochanomizu University, 2-1-1 Otsuka, Bunkyo-ku, Tokyo, Japan}
\email{tsang.sin.yi@ocha.ac.jp}
\urladdr{http://sites.google.com/site/cindysinyitsang/}



\begin{abstract}The purpose of this paper is to  explore possible definitions of $n$-isoclinism for skew braces. We also introduce the notions of verbal sub-skew braces and marginal left ideals.
\end{abstract}

\maketitle

\vspace{-4mm}

\section{Introduction}

Let us first recall the definition of $n$-isoclinism for groups that was introduced by Hall \cite{Hall,Hall'} (also see \cite{Bioch,Hekster,van der Waall}). For any group $G$, write
\[ \gamma_1(G) = G,\quad \gamma_{n+1}(G)=[G,\gamma_n(G)]=[\gamma_n(G),G]\,\ (n\geq 1)\]
for its lower central series, and write
\[ \zeta_0(G)=1,\quad Z(G/\zeta_n(G)) = \zeta_{n+1}(G)/\zeta_n(G)\,\ (n\geq 0)\]
for its upper central series. For any groups $G$ and $H$, a pair
\[ \xi : G/\zeta_n(G)\longrightarrow H/\zeta_n(H),\quad \theta:\gamma_{n+1}(G)\longrightarrow \gamma_{n+1}(H)\]
of isomorphisms is said to be an \textit{$n$-isoclinism} if the diagram
\begin{equation}\label{group:diagram}
\begin{tikzcd}[column sep = 2.5cm, row sep = 1.5cm]
 (G/\zeta_n(G))^{\oplus n+1} \arrow{r}{\phi(G)} \arrow{d}[left]{\xi^{\oplus n+1}}& \gamma_{n+1}(G)\arrow{d}{\theta}\\
(H/\zeta_n(H))^{\oplus n+1} \arrow{r}[below]{\phi(H)} & \gamma_{n+1}(H)
\end{tikzcd}
\end{equation}
commutes, where the horizontal maps are induced by
\[  \phi(-):(x_1,x_2,\dots,x_{n+1}) \mapsto [x_1,[x_2,[\cdots,[x_n,x_{n+1}]]\cdots]].\]
We can also use left-normed commutator for the horizontal maps and there is no fundamental difference (see Section \ref{sec:verbal} for details). We say that $G$ and $H$ are \textit{$n$-isoclinic} if there is an $n$-isoclinism from $G$ to $H$.

\smallskip 

There are two important facts to note here. For any group $G$, it is well-known that the following hold (or see Propositions \ref{prop:verbal} and \ref{prop:marginal}).
\begin{enumerate}[(1)]
\item The $(n+1)$st commutator $\gamma_{n+1}(G)$ is the \textit{verbal subgroup} of the word
\[[x_1,[x_2,[\cdots,[x_n,x_{n+1}]]\cdots]],\]
namely $\gamma_{n+1}(G)$ is generated by
\[[g_1,[g_2,[\cdots,[g_n,g_{n+1}]]\cdots]]\]
for $g_1,g_2,\dots,g_n,g_{n+1}$ ranging over $G$ (see Definition \ref{def:verbal}).
\smallskip
\item The $n$th center $\zeta_n(G)$ is the \textit{marginal subgroup} of the word
\[[x_1,[x_2,[\cdots,[x_n,x_{n+1}]]\cdots]],\]
namely $\zeta_n(G)$ is the largest subgroup of $G$ for which the above word is well-defined on the left and right cosets (see Definition \ref{def:marginal}).
\end{enumerate}
Fact (1) implies that an $n$-isoclinism naturally induces an $(n+1)$-isoclinism. This means that two groups are $(n+1)$-isoclinic when they are $n$-isoclinic, which is a desirable property that one would expect $n$-isoclinism to satisfy. Fact (2) ensures that the horizontal maps of \eqref{group:diagram} are indeed well-defined.

\smallskip

Let us now consider possible definitions of $n$-isoclinism in the context of skew braces. Recall that a \textit{skew (left) brace} is any set $A = (A,\cdot,\circ)$ endowed with two group operations $\cdot$ and $\circ$ such that the left brace relation
\[ a\circ (b\cdot c) = (a\circ b)\cdot a^{-1}\cdot (a\circ c)\]
holds for all $a,b,c\in A$. Here, for each $a\in A$, we write $a^{-1}$ for its inverse in $(A,\cdot)$ and $\overline{a}$ for its inverse in $(A,\circ)$. Also, we use $1$ to denote the common identity element of $(A,\cdot)$ and $(A,\circ)$. It is well-known that
\begin{equation}\label{eqn:lambda} \lambda : (A,\circ)\longrightarrow \Aut(A,\cdot);\quad a\mapsto \lambda_a := (b\mapsto a^{-1}\cdot (a\circ b)) \end{equation}
is a homomorphism, and the two group operations are linked by formulae
\[ a\circ b = a\cdot \lambda_a(b),\quad a\cdot b = a\circ \lambda_{\overline{a}}(b),\quad \overline{a} = \lambda_{\overline{a}}(a^{-1}),\quad a^{-1} = \lambda_a(\overline{a}).\]
In addition to the commutators
\[ [a,b] = a\cdot b\cdot a^{-1}\cdot b^{-1},\quad [a,b]_\circ = a\circ b \circ \overline{a}\circ \overline{b}\]
of the groups $(A,\cdot)$ and $(A,\circ)$, it is common to consider the star product
\[ a * b = a^{-1}\cdot (a\circ b)\cdot b^{-1} = \lambda_a(b)\cdot b^{-1},\]
which quantifies the difference between the two group operations attached to the skew brace $A$. For any subsets $X$ and $Y$ of $A$, let $X*Y$ denote the subgroup of $(A,\cdot)$ generated by the elements $x*y$ with $x\in X$ and $y\in Y$.

\smallskip

For any skew brace $A$, the \textit{left series} of $A$ is defined by
\[ A^1 = A,\quad A^{n+1} = A*A^{n}\,\ (n\geq 1),\]
and the \textit{right series} of $A$ is defined by
\[ A^{(1)} = A,\quad A^{(n+1)} = A^{(n)}*A \,\ (n\geq 1).\]
They may be regarded as analogs of the lower central series. From a categorical point of view \cite{SKB}, the series introduced in \cite{central} seems to be the more appropriate analog, but we shall not consider it in this paper. 

\smallskip

For any skew brace $A$, the \textit{annihilator} of $A$ is defined to be
\begin{align*}
    \Ann(A) & = \{z\in A\mid [a,z] = [a,z]_\circ = z * a =1\mbox{ for all } a\in A\}\\
    & = \{z\in A \mid [a,z] =z*a = a*z=1\mbox{ for all }a\in A\},
\end{align*}
which is a natural analog of the center. It is well-known that $\Ann(A)$ is an ideal of $A$, namely it is a normal subgroup of both $(A,\cdot)$ and $(A,\circ)$ that is invariant under $\mathrm{Im}(\lambda)$. The \textit{annihilator series} of $A$ is defined by
\[ \Ann_0(A) = 1,\quad \Ann(A/\Ann_n(A)) = \Ann_{n+1}(A)/\Ann_n(A)\,\ (n\geq 0).\]
This may be regarded as an analog of the upper central series. Notice that $\Ann_n(A)$ is an ideal of $A$ by the lattice isomorphism theorem of groups, so we can indeed form the skew brace quotient $A/\Ann_n(A)$. 

\smallskip

In \cite{isoclinism}, Letourmy and Vendramin defined isoclinism (i.e. $1$-isoclinism) for skew braces, as follows. For any skew brace $A$, the \textit{commutator} of $A$ is the subgroup $A'$ of $(A,\cdot)$ generated by $\gamma_2(A,\cdot)$ and $A^2 = A*A = A^{(2)}$. For any skew braces $A$ and $B$, a pair 
\[ \xi : A/\Ann(A)\longrightarrow B/\Ann(B),\quad \theta: A' \longrightarrow B' \]
of isomorphisms is said to be an \textit{isoclinism} if the diagram
\[\begin{tikzcd}[column sep = 2.5cm, row sep = 1.5cm]
(A/\Ann(A))^{\oplus 2}\arrow{r}{\phi_i(A)} \arrow{d}[left]{\xi^{\oplus 2}}& A'\arrow{d}{\theta}\\
(B/\Ann(B))^{\oplus 2} \arrow{r}[below]{\phi_i(B)} & B'
\end{tikzcd}\]
commutes for both $i=1,2$, where the horizontal maps are induced by
\[ \begin{cases}
    \phi_1(-) : (x_1,x_2)\mapsto [x_1,x_2], \\ \phi_2(-):(x_1,x_2)\mapsto x_1*x_2.
\end{cases}\]
This definition has many nice properties that are analogous to groups, such as the existence of a stem skew brace in every isoclinism class \cite[Theorem 2.18]{isoclinism} and the uniqueness of its order \cite[Proposition 2.17]{isoclinism}.

\smallskip

It is natural to ask how $n$-isoclinism ought to be defined for skew braces for an arbitrary $n$. This question has been investigated in \cite{nisoclinism}. But there are some issues that are worth paying attention to. For example:
\begin{enumerate}[(1)]
\item As a sub-skew brace, say, it is unclear whether the $n$th term $A^n$ in the left series is generated by the $n$-fold right-normed star products
\[a_1 * (a_2 * ( \cdots * (a_{n-1}*a_n)\cdots )),\]
and similarly, it is unclear whether the $n$th term $A^{(n)}$ in the right series is generated by the $n$-fold left-normed star products
\[(( \cdots (a_1*a_2) * \cdots )*a_{n-1}) * a_n.\]
See Section \ref{subsec:verbal sub'} for details.
\smallskip
\item Modulo the $n$th annihilator $\Ann_n(A)$, the $(n+1)$-fold right-normed star product is always well-defined, but the  $(n+1)$-fold left-normed star product is not well-defined in general. See Section \ref{subsec:marginal sub'} for details.
\end{enumerate}
If we simply replace $A*A$ by $A^{n+1}$ or $A^{(n+1)}$ and $\Ann(A)$ by $\Ann_n(A)$, two skew braces might fail to be $(n+1)$-isoclinic even when they are $n$-isoclinic by point (1), and the notion of $n$-isoclinism would only apply to some skew braces when the left-normed star product is involved by point (2). To avoid these problems, perhaps the correct approach to $n$-isoclinism of skew braces is to use analogs of the verbal and marginal subgroups. 

\smallskip

The purpose of this paper is to introduce verbal sub-skew braces (Definition \ref{def:verbal'}) and marginal left ideals (Definition \ref{def:marginal'}) in the framework of skew braces. In the case of groups, a \textit{word}, in the variables $x_1,\dots,x_n$ say, is an expression that can be written as a product of
\[ x_1,\, x_2,\, \dots \, ,\, x_n,\, x_1^{-1},\, x_2^{-1},\, \dots\, ,\, x_n^{-1}.\]
In the case of skew braces, there are two group operations involved, so we define a \textit{word}, in the variables $x_1,\dots,x_n$ say, to be an expression that can be written as a product, using any combination of $\cdot$ and $\circ$, of
\[ x_1,\, x_2,\, \dots \, ,\, x_n,\, x_1^{-1},\, x_2^{-1},\, \dots\, ,\, x_n^{-1},\, \overline{x_1},\, \overline{x_2},\, \dots\, ,\, \overline{x_n}.\]
Due to their relation with $n$-isoclinism, we are particularly interested in the left-normed and right-normed star products. We shall investigate the verbal sub-skew braces and marginal left ideals associated to these words. We hope that our discussion will provide new insights into how $n$-isoclinism ought to be defined for skew braces (see Section \ref{last section}).

\section{Verbal subgroups and marginal subgroups}\label{sec:verbal}

We have the word in two variables given by the commutator
\[ [x_1,x_2] = x_1\cdot ({^{x_2}}x_1)^{-1}, \mbox{ where }{^{x_2}}x_1 = x_2\cdot x_1\cdot x_2^{-1}.\]
Recursively, we can define the $n$-fold left-normed commutator
\[\llbracket x_1,x_2,\dots,x_n\rrbracket   = [ \llbracket x_1,\dots,x_{n-1}\rrbracket,x_n]\]
and the $n$-fold right-normed commutator
\[[x_1,x_2,\dots,x_n]  = [x_1,[x_2,\dots,x_{n}]]\]
for any $n\geq 3$. We shall be particularly interested in these two words.

\subsection{Verbal subgroups}\label{subsec:verbal sub}

The next definition is due to Hall \cite{Hall'}.

\begin{definition}\label{def:verbal}
Let $W$ be a collection of group words in $n$ variables. For any group $G$, the \textit{verbal subgroup} of $G$ associated to $W$ is defined to be the subgroup $V_W(G)$ of $G$ generated by
\[ w(g_1,g_2,\dots,g_n)\]
for $w$ ranging over $W$ and $g_1,g_2,\dots,g_n$ ranging over $G$.
\end{definition}

The next propositions are well-known.

\begin{prop}\label{prop:verbal} Let $G$ be a group. For any $n\geq 1$, we have
\[ \gamma_n(G) = V_{[x_1,x_2,\dots,x_n]}(G).\]
\end{prop}
\begin{proof}
The case $n=1$ is trivial. Now, suppose that $\gamma_n(G)$ is generated by
\[ S:=\{ [ g_1,g_2,\dots,g_n] : g_1,g_2,\dots,g_n\in G\},\]
and we want to show that $\gamma_{n+1}(G)$ is generated by
\[ T:=\{ [g_1,g_2,\dots,g_{n+1}]: g_1,g_2,\dots,g_{n+1}\in G\}.\]
It is enough to check that $[g,h] \in \langle T\rangle$ for all $g\in G$ and $h\in \gamma_n(G)$. By the induction hypothesis, we can write $h=s_1's_2'\cdots s_\ell'$, where 
\[ s_k'=s_k^{\epsilon_k} \mbox{ for some }s_k\in S,\, \epsilon_k\in \{\pm 1\}\]
for each $1\leq k\leq \ell$. The key here is that we have the identities
\begin{align}\label{eqn:commutator1}
   &[x,y_1y_2]=  [x,y_1][{^{y_1}}x,{^{y_1}}y_2],\,\ [x,y^{-1}] = [{^{y^{-1}}}x,y]^{-1},\\\label{eqn:commutator2}
   &{^y}[x_1,x_2,\dots,x_n]= [{^y}x_1,{^y}x_2,\dots,{^y}x_n].
\end{align}
For $\ell=1$, we then see that
\[ [g,s_1] \in T,\quad [g,s_1^{-1}]= [{^{s_1^{-1}}}g,s_1]^{-1} \in \langle T\rangle,\]
and so $[g,s_1']\in \langle T\rangle$. For $\ell\geq 2$, similarly we can use \eqref{eqn:commutator1} to write
\begin{align*}
 [g,s_1's_2'\cdots s_\ell'] & = [g,s_1'][{^{s_1'}g},{^{s_1'}(s_2'\cdots s_\ell')}]\\
 & = [g,s_1'][{^{s_1'}g}, ({^{s_1'}s_2})^{\epsilon_2}\cdots ({^{s_1'}s_\ell})^{\epsilon_\ell}].\end{align*}
Notice that ${^{s_1'}s_2},\dots,{^{s_1'}s_\ell}\in S$ by \eqref{eqn:commutator2}, and so it follows from induction on $\ell$ that $[g,h]\in \langle T\rangle$, as desired.
\end{proof}

\begin{prop}\label{prop:verbal'}Let $G$ be a group. For any $n\geq 1$, we have
\[ \gamma_n(G) = V_{\llbracket x_1,x_2,\dots,x_n \rrbracket}(G).\]
\end{prop}
\begin{proof}Instead of \eqref{eqn:commutator1} and \eqref{eqn:commutator2}, we use the identities
\begin{align}\label{eqn:commutator'1}
   &[y_1y_2,x]=  [{^{y_1}}y_2,{^{y_1}}x][y_1,x],\,\ [y^{-1},x] = [y,{^{y^{-1}}}x]^{-1},\\\label{eqn:commutator'2}
   &{^y}\llbracket x_1,x_2,\dots,x_n\rrbracket= \llbracket {^y}x_1,{^y}x_2,\dots,{^y}x_n\rrbracket,
\end{align}
and the proof is essentially the same as Proposition \ref{prop:verbal}. 
\end{proof}

Proposition \ref{prop:verbal} is a  crucial fact in order for an $n$-isoclinism to induce an $(n+1)$-isoclinism. This implies that the notion of $n$-isoclinism gets weaker as $n$ increases, which is a natural desirable property.

\begin{prop}\label{prop:nimpliesn+1}Two $n$-isoclinic groups are always $(n+1)$-isoclinic.
\end{prop}
\begin{proof}
Let $G$ and $H$ be $n$-isoclinic groups and let $(\xi,\theta)$ be any $n$-isoclinism from $G$ to $H$. From the natural isomorphisms
\[
 G/\zeta_{n+1}(G) \simeq \frac{G/\zeta_n(G)}{Z(G/\zeta_n(G))},\quad 
H/\zeta_{n+1}(H) \simeq \frac{H/\zeta_n(H)}{Z(H/\zeta_{n}(H))},
\]
we  see that $\xi$ induces an isomorphism $\xi' : G/\zeta_{n+1}(G)\longrightarrow H/\zeta_{n+1}(H)$. For any $g\in G$, let $\widetilde{g}$ denote the class of $g$ modulo $\zeta_{n}(G)$. By the commutativity of \eqref{group:diagram}, for any $g_1,g_2,\dots,g_{n+2}\in G$, we have
\begin{align*}
\theta( [g_1,g_2,\dots,g_{n+2}]) 
& = [\xi(\widetilde{g_1}),\dots,\xi(\widetilde{g_n}),\xi([\widetilde{g_{n+1}},\widetilde{g_{n+2}}])]\\
&=[\xi(\widetilde{g_1}),\dots,\xi(\widetilde{g_n}),\xi(\widetilde{g_{n+1}}),\xi(\widetilde{g_{n+2}})] \in \gamma_{n+2}(H).
\end{align*} 
Since the elements $[g_1,g_2,\dots,g_{n+2}]$ generate $\gamma_{n+2}(G)$ by Proposition \ref{prop:verbal}, it follows that $\theta$ restricts to an isomorphism $\theta' : \gamma_{n+2}(G)\longrightarrow \gamma_{n+2}(H)$. Clearly $(\xi',\theta')$ is an $(n+1)$-isoclinism, whence $G$ and $H$ are $(n+1)$-isoclinic. 
\end{proof}

\subsection{Marginal subgroups}

The next definition is again due to Hall \cite{Hall'}.

\begin{definition}\label{def:marginal}
Let $W$ be a collection of group words in $n$ variables. For any group $G$, the \textit{marginal subgroup} of $G$ associated to $W$ is defined to be the subgroup $M_W(G)$ of $G$ consisting of the $z\in G$ such that
\begin{align*}
    w(g_1,g_2,\dots,g_n) & = w(g_1,\dots,g_iz,\dots, g_n)\\
    & = w(g_1,\dots,zg_i,\dots, g_n)
\end{align*}
for all $w\in W$, $g_1,\dots,g_n\in G$, and $1\leq i\leq n$. It is straightforward to verify that $M_W(G)$ is indeed a subgroup of $G$.
\end{definition}

The next propositions are well-known.

\begin{prop}\label{prop:marginal}
Let $G$ be a group. For any $n\geq 1$, we have
\[ \zeta_n(G) = M_{ [x_1,x_2,\dots,x_{n+1}]}(G).\]
\end{prop}
\begin{proof}The case $n=1$ is trivial. Now, suppose that the claim is true up to $n-1$. Put $w = [x_1,x_2,\dots,x_{n+1}]$ for short. If $z\in M_w(G)$, then
\begin{align*} 
&\hspace{1.05cm}[g_1,g_2,\dots,g_n,z] = 1\mbox{ for all }g_1,g_2,\dots,g_n\in G\\
&\implies [g_2,\dots,g_n,z]\in Z(G) \mbox{ for all }g_2,\dots,g_n\in G\\
&\implies [g_3,\dots,g_n,z] \in \zeta_2(G) \mbox{ for all }g_3,\dots,g_n\in G\\[-2pt]
& \hspace{5cm}\vdots\\
&\implies [g_n,z]\in \zeta_{n-1}(G)\mbox{ for all }g_n\in G\\
&\implies z\in \zeta_{n}(G).
\end{align*}
Conversely, if $z\in \zeta_{n}(G)$, then we can show that $z\in M_w(G)$, as follows. We let $g_1,g_2,\dots,g_{n+1}\in G$ be arbitrary. It suffices to show that
\[ [g_1,g_2,\dots,g_{n+1}] = [g_1,\dots,g_iz,\dots,g_{n+1}]\]
for all $1\leq i\leq n+1$ because $\zeta_{n}(G)$ is normal in $G$.
\begin{enumerate}[(a)]
\item \underline{The case $i=n+1$:} Using \eqref{eqn:commutator1}, we can write
\[ [g_1,\dots,g_{n},g_{n+1}z] = [g_1,\dots,g_{n-1},[g_{n},g_{n+1}][{^{g_{n+1}}}g_{n},{^{g_{n+1}}}z]].\]
Since $z\in \zeta_{n}(G)$ and $\zeta_{n}(G)$ is normal in $G$, we have
\[ [{^{g_{n+1}}}g_{n},{^{g_{n+1}}}z] \in \zeta_{n-1}(G).\]
But $\zeta_{n-1}(G)\subseteq M_{[x_1,x_2,\dots,x_{n}]}(G)$ by induction, so it follows that
\[ [g_1,\dots,g_n,g_{n+1}z] = [g_1,\dots,g_{n-1},[g_{n},g_{n+1}]],\] 
which is as desired.
\smallskip
\item \underline{The case $1\leq i\leq n$:} Put $g = [g_{i+1},\dots,g_{n+1}]$. Using \eqref{eqn:commutator'1}, we can write
\[  [g_1,\dots,g_iz,\dots,g_{n+1}] = [g_1,\dots,g_{i-1}, [{^{g_i}}z,{^{g_i}}g][g_i,g]].\]
As shown in \cite[5.1.11(iii)]{Robinson book}, for example, we have the inclusion
\begin{equation}\label{eqn:inclusion} [\zeta_n(G),\gamma_{n-i+1}(G)]\subseteq \zeta_{i-1}(G).\end{equation}
Since $z\in \zeta_n(G)$ and $g\in \gamma_{n-i+1}(G)$, and both of the subgroups here are normal in $G$, we have the containment
\[ [{^{g_i}}z,{^{g_i}}g] \in \zeta_{i-1}(G).\]
But $\zeta_{i-1}(G) \subseteq M_{[x_1,x_2,\dots,x_i]}(G)$ by induction, so it follows that
\[  [g_1,\dots,g_iz,\dots,g_{n+1}] = [g_1,\dots,g_{i-1}, [g_i,g]],\]
which is as desired.
\end{enumerate}
This completes the proof.
\end{proof}

\begin{prop}\label{prop:marginal'}
Let $G$ be a group. For any $n\geq 1$, we have
\[ \zeta_n(G) = M_{ \llbracket x_1,x_2,\dots,x_{n+1} \rrbracket}(G).\]
\end{prop}
\begin{proof}
It is completely analogous to Proposition \ref{prop:marginal}.
\end{proof}

Proposition \ref{prop:marginal} shows that the horizontal maps in \eqref{group:diagram} are well-defined so that the notion of $n$-isoclinism applies to all pairs of groups.

\subsection{Definition of $n$-isoclinism} From Propositions \ref{prop:verbal}, \ref{prop:verbal'}, \ref{prop:marginal}, and \ref{prop:marginal'}, we see that left-normed commutators could be used for the horizontal maps in \eqref{group:diagram} and there is no fundamental difference. The commutativity of \eqref{group:diagram} does not depend on this choice either. Indeed, we have the identities 
\[ [x,y] = [y^{-1},{^y}x],\quad [y,x] = [{^y }x,y^{-1}]\]
by replacing $y$ by $y^{-1}$ in \eqref{eqn:commutator1} and \eqref{eqn:commutator'1}, respectively. Using these, we easily verify by induction that for $n$ even and $n$ odd, respectively, we have
\begin{align*}
& [ x_1,x_2,\dots,x_{n+1}] \\
&= \begin{cases}
    \llbracket x_{n+1},x_{n},{^{[x_{n},x_{n+1}]}}x_{n-1},\dots,x_4,{^{[x_4,\dots,x_{n+1}]}x_3}, x_2,{^{[x_2,\dots, x_{n+1}]} x_1}\rrbracket,\\
  \llbracket x_n,x_{n+1},x_{n-1},{^{[x_{n-1},x_n,x_{n+1}]}}x_{n-2},\dots,x_4,{^{[x_4,\dots,x_{n+1}]}x_3}, x_2,{^{[x_2,\dots, x_{n+1}]} x_1}\rrbracket.
\end{cases}
\end{align*}
Similarly, for $n$ even and $n$ odd, respectively, we have
\begin{align*}
&  \llbracket x_1,x_2,\dots,x_{n+1}\rrbracket\\
&= \begin{cases}
 [ {^{\llbracket x_1,\dots,x_{n}\rrbracket }x_{n+1}}, x_{n},{^{\llbracket x_1,\dots,x_{n-2}\rrbracket }x_{n-1}},x_{n-2},\dots,{^{\llbracket x_1,x_2\rrbracket}}x_3,x_2,x_1]
  ,\\
    [ {^{\llbracket x_1,\dots,x_{n}\rrbracket }x_{n+1}},x_{n},{^{\llbracket x_1,\dots,x_{n-2}\rrbracket }x_{n-1}},x_{n-2},\dots,{^{\llbracket x_1,x_2,x_3\rrbracket}}x_4,x_3,x_1,x_2].
\end{cases}
\end{align*}
It follows immediately from these identities that the commutativity of \eqref{group:diagram} does not depend on whether we use right-normed or left-normed commutators for the horizontal maps.

\section{Verbal sub-skew braces and marginal left ideals}

We have the word in two variables given by the star product
\[ x_1 * x_2 = x_1^{-1}\cdot (x_1\circ x_2)\cdot x_2^{-1}.\]
Recursively, we can define the $n$-fold left-normed star product
\[ x_1\star x_2\star \cdots \star x_n = (x_1\star \cdots \star x_{n-1}) * x_n\]
and the $n$-fold right-normed star product
\[ x_1* x_2* \cdots *x_n = x_1 * (x_2 * \cdots * x_n)\]
for any $n\geq 3$. We shall be particularly interested in these two words, and also the group commutators 
\[ [x_1,x_2,\dots,x_n]\quad \mbox{and}\quad  [x_1,x_2,\dots,x_n]_\circ\]
with respect to $\cdot$ and $\circ$, respectively. There is no fundamental difference, so we shall always take our commutators to be right-normed.

\subsection{Verbal sub-skew braces}\label{subsec:verbal sub'}

We propose the following definition, which is an analog of Definition \ref{def:verbal} in the context of skew braces.

\begin{definition}\label{def:verbal'}
Let $W$ be a collection of skew brace words in $n$ variables. For any skew brace $A$, the \textit{verbal sub-skew brace} of $A$ associated to $W$ is defined to be the sub-skew brace $V_W(A)$ of $A$ generated by
\[ w(a_1,a_2,\dots,a_n)\]
for $w$ ranging over $W$ and $a_1,a_2,\dots,a_n$ ranging over $A$.
\end{definition}

It is natural to ask whether the analogs of Propositions \ref{prop:verbal} and \ref{prop:verbal'} hold for skew braces. In other words, is it true that 
\begin{equation}\label{eqn:A verbal} A^n = V_{x_1*x_2*\cdots *x_n}(A)\quad \mbox{and}\quad A^{(n)} = V_{x_1\star x_2\star\cdots \star x_n}(A)\end{equation}
for all $n\in \mathbb{N}$ and for all skew braces $A$? The cases $n=1,2$ are trivial, but it seems that the proof of Proposition \ref{prop:verbal} cannot be modified to prove the case when $n\geq 3$. For the right-normed star product, we have
\begin{equation}\label{eqn:star1} 
x * (y_1y_2) = (x*y_1)  y_1(x*y_2) y_1^{-1},\,\ x*y^{-1} = y^{-1}(x*y)^{-1}y,
\end{equation}
which is an analog of \eqref{eqn:commutator1}. For the left-normed star product, we have
\begin{equation}\label{eqn:star2} 
(y_1y_2)*x= y_2^{-1}(y_1*x)y_2(y_2*x),\,\ y^{-1}*x = y(y*x)^{-1}y^{-1},
\end{equation}
which is an analog of \eqref{eqn:commutator2}, only in some special skew braces. For example, the identity \eqref{eqn:star2} is true in a \textit{two-sided} skew brace, namely a skew brace $A$ for which the right brace relation
\[ (b\cdot c) \circ a = (b\circ a)\cdot a^{-1}\cdot (c\circ a) \]
also holds for all $a,b,c\in A$. But in general, we only have the identity
\begin{equation}\label{eqn:star3} (y_1\circ y_2) * x = (y_1*(y_2*x)) (y_2*x)(y_1*x),\end{equation}
which is not easy to work with. In any case, the analogs of \eqref{eqn:commutator2} and \eqref{eqn:commutator'2} both fail because we cannot simplify $y(x_1*x_2)y^{-1}$ to a single star product in general, and this is why we cannot adapt the proof of Proposition \ref{prop:verbal} to prove \eqref{eqn:A verbal}. In view of this, we introduce two new series for a skew brace.

\begin{definition} For a skew brace $A$, define:
\begin{enumerate}[(1)]
\item the \textit{verbal left series} of $A$ to be $A^{\underline{n}} = V_{x_1*x_2*\cdots *x_n}(A)$ for $n\geq 1$;
\item the \textit{verbal right series} of $A$ to be $A^{(\underline{n})} = V_{x_1\star x_2\star\cdots \star x_n}(A)$ for $n\geq 1$.
\end{enumerate}
Clearly $A^{\underline{n}} = A^n = A^{(n)} = A^{(\underline{n})}$ for $n=1,2$, while $A^{\underline{n}}\subseteq A^n$ and $A^{(\underline{n})} \subseteq A^{(n)}$ for $n\geq 3$. We do not know whether the latter inclusions can be strict.
\end{definition}

It is known that the terms in the left series are all left ideals, and we now prove that the same is true for the terms in the verbal left series. It is also known that the terms in the right series are all ideals, but we do not know whether the terms in the verbal right series are even left ideals.

\begin{prop}\label{prop:verbal series} Let $A$ be a skew brace. For any $n\geq 1$, the sub-skew brace $A^{\underline{n}}$ is a left ideal of $A$ and is generated by $a_1*\cdots *a_n$ for $a_1,\dots,a_n\in A$ as a subgroup of $(A,\cdot)$.
\end{prop}
\begin{proof}
Consider the subgroup $I_n$ of $(A,\cdot)$ that is generated by $a_1*\cdots *a_n$ for $a_1,\dots,a_n\in A$. It is enough to show that $I_n$ is a left ideal of $A$, because then $I_n$ would be a sub-skew brace of $A$, which clearly means that $I_n = A^{\underline{n}}$ has to hold. To that end, we use the identity
\[ \lambda_y(x_1*x_2) = (y\circ x_1\circ \overline{y}) * \lambda_y(x_2). \]
Since the star product is right-normed here, applying this repeatedly yields
\[ \lambda_b(a_1*\cdots *a_n) = (b\circ a_1\circ \overline{b}) * (b\circ a_2\circ \overline{b})*\cdots *(b\circ a_{n-1}\circ \overline{b})*\lambda_b(a_n)\]
for all $b,a_1,\dots,a_n\in A$. Since these elements $a_1*\cdots *a_n$ generate $I_n$ with respect to $\cdot$ and $\lambda_b\in \Aut(A,\cdot)$, this implies that $\lambda_b(I_n)\subseteq I_n$ for all $b\in A$. Thus, indeed $I_n$ is a left ideal of $A$.
\end{proof}

Recall that the proof of Proposition \ref{prop:verbal} cannot be generalized to prove a similar statement for the star product for the following reasons:
\begin{enumerate}[(1)]
\item For the right-normed star product, the analog of \eqref{eqn:commutator2} fails.
\item For the left-normed star product, the analogs of \eqref{eqn:commutator'1} and \eqref{eqn:commutator'2} fail.
\end{enumerate}
There are conditions on a skew brace $A$ under which these issues vanish, in which case we can show that $A^{n} = A^{\underline{n}}$ and $A^{(n)} = A^{(\underline{n})}$, respectively.

\smallskip

First, we consider the right-normed star product.

\begin{lem}\label{lem:conjugation} Let $A$ be any skew brace such that the image of \eqref{eqn:lambda} is normalized by $\mathrm{Inn}(A,\cdot)$. Then for any $b,a_1,a_2,\dots,a_n\in A$, we have
\[ b(a_1*a_2*\cdots * a_n)b^{-1} = a_1' * a_2' * \cdots * a_n'\]
for some $a_1',a_2',\dots,a_n'\in A$.
\end{lem}
\begin{proof} For any $a\in A$, the hypothesis implies that
\[ \mathrm{conj}(b) \lambda_a \mathrm{conj}(b)^{-1} = \lambda_{a'}\]
for some $a'\in A$, where $\mathrm{conj}(b)$ denotes the inner automorphism $x\mapsto bxb^{-1}$ of $(A,\cdot)$. Thus, for any $a,c\in A$, we can write
\begin{align}\label{eqn:conjugation*}
b(a*c) b^{-1} & = \mathrm{conj}(b)( \lambda_a(c)c^{-1})\\\notag
& =  \lambda_{a'}(\mathrm{conj}(b)(c)) \cdot \mathrm{conj}(b)(c)^{-1}\\\notag
& = \lambda_{a'}(bcb^{-1})\cdot (bcb^{-1})^{-1}\\\notag
& = a' * (bcb^{-1}).
\end{align}
Since the star product is right-normed here, we deduce that
\begin{align*}
b(a_1*a_2*\cdots * a_n)b^{-1} & = a_1' * (b (a_2*\cdots *a_n)b^{-1})\\
    & = a_1' * a_2' * (b(a_3*\cdots *a_n)b^{-1})\\
    & \hspace{2cm}\vdots\\
    & = a_1' * a_2' * \cdots * a_{n-1}' * (ba_nb^{-1})
\end{align*}
for some $a_1',a_2',\dots,a_{n-1}'\in A$, as desired.
\end{proof}

\begin{prop}\label{prop:equal1} Let $A$ be any skew brace such that the image of \eqref{eqn:lambda} is normalized by $\mathrm{Inn}(A,\cdot)$. Then $A^n = A^{\underline{n}}$ for all $n\geq 1$.
\end {prop}
\begin{proof}
The case $n=1$ is trivial. Now, suppose that $A^n = A^{\underline{n}}$. By Proposition \ref{prop:verbal series}, this means that $A^n$ is generated by
\[ S:=\{ a_1* a_2* \cdots * a_n : a_1,a_2,\dots,a_n\in A\}\]
as a subgroup of $(A,\cdot)$, and we shall show that $A^{n+1}$ is generated by
\[ T:=\{ a_1 * a_2*\cdots *a_{n+1}: a_1,a_2,\dots,a_{n+1}\in A\}\]
as a subgroup of $(A,\cdot)$. It suffices to show that $a*b \in \langle T\rangle$ for all $a\in A$ and $b\in A^n$, where $\langle\,\ \rangle$ denotes subgroup generation in $(A,\cdot)$. By the induction hypothesis, we can write $b=s_1's_2'\cdots s_\ell'$, where 
\[ s_k'=s_k^{\epsilon_k} \mbox{ for some }s_k\in S,\, \epsilon_k\in \{\pm 1\}\]
for each $1\leq k\leq \ell$. For $\ell=1$, it follows from \eqref{eqn:star1} and \eqref{eqn:conjugation*} that
\[ a*s_1 \in T,\quad a*s_1^{-1}= s_1^{-1}(a*s_1)^{-1}s_1 = (a'*s_1)^{-1}\in \langle T\rangle\]
for some $a'\in A$, and so $a*s_1'\in \langle T\rangle$. For $\ell\geq 2$, again \eqref{eqn:star1} and \eqref{eqn:conjugation*} imply
\begin{align*}
 a * (s_1's_2'\cdots s_\ell') & = (a*s_1')s_1'(a * (s_2'\cdots s_\ell'))s_1'^{-1}\\
 & = (a*s_1') (a'' * ({^{s_1'}s_2})^{\epsilon_2}\cdots ({^{s_1'}s_\ell})^{\epsilon_\ell})\end{align*}
for some $a''\in A$. But note that ${^{s_1'}s_2},\dots,{^{s_1'}s_\ell}\in S$ by Lemma \ref{lem:conjugation}. Hence, it follows from induction on $\ell$ that $a*b\in \langle T\rangle$, as desired.
\end{proof}

Next, we consider the left-normed star product. The proof is very similar to Proposition \ref{prop:equal1}. In fact, it is slightly easier because we no longer have to worry about conjugation of star products when $A*A\subseteq Z(A,\cdot)$.

\begin{prop}\label{prop:equal2} Let $A$ be any two-sided skew brace such that $A*A$ lies in $Z(A,\cdot)$. Then $A^{(n)} = A^{(\underline{n})}$ for all $n\geq 1$. In fact, the sub-skew brace $A^{(n)}$ is generated by $a_1\star \cdots \star a_n$, where $a_1,\dots,a_n$ ranges over $A$, not just as a sub-skew brace of $A$ but also as a subgroup of $(A,\cdot)$, for all $n\geq 1$.
\end{prop}
\begin{proof} The case $n=1$ is trivial. Now, suppose that $A^{(n)}$ is generated by
\[ S:=\{ a_1 \star a_2\star \cdots \star a_n: a_1,a_2,\dots,a_n\in A\}\]
as a subgroup of $(A,\cdot)$, and we shall show that $A^{(n+1)}$ is generated by
\[ T:=\{a_1\star a_2\star \cdots \star a_{n+1} : a_1,a_2,\dots,a_{n+1}\in A\}\]
as a subgroup of $(A,\cdot)$. It suffices to show that $b*a\in \langle T\rangle$ for all $a\in A$ and $b\in A^{(n)}$, where $\langle \,\ \rangle$ denotes subgroup generation in $(A,\cdot)$. By the induction hypothesis, we can write $b=s_1's_2'\cdots s_\ell'$, where 
\[ s_k'=s_k^{\epsilon_k} \mbox{ for some }s_k\in S,\, \epsilon_k\in \{\pm 1\}\]
for each $1\leq k\leq \ell$. Since $A$ is two-sided and $A*A$ lies in $Z(A,\cdot)$, it follows from \eqref{eqn:star2} that $x\mapsto x * a$ defines an endomorphism on $(A,\cdot)$, so we have
\begin{align*}
(s_1's_2'\cdots s_\ell')* a
 = (s_1*a)^{\epsilon_1}(s_2*a)^{\epsilon_2}\cdots (s_\ell *a)^{\epsilon_\ell}.
\end{align*}
Since $s_k * a\in T$ for each $1\leq k\leq \ell$, we see that $b*a\in \langle T\rangle$, as desired.
 \end{proof}

\subsection{Marginal left ideals}\label{subsec:marginal sub'}

We propose the following definition, which is an analog of Definition \ref{def:marginal} in the context of skew braces.
 
\begin{definition}\label{def:marginal'}
Let $W$ be a collection of skew brace words in $n$ variables. For any skew brace $A$, the \textit{marginal left ideal} of $A$ associated to $W$ is defined to be the left ideal $M_W(A)$ of $A$ consisting of the $z\in A$ such that
\begin{align}\label{eqn:w1}
     w(a_1,a_2,\dots,a_n) & = w(a_1,\dots,a_i\lambda_b(z),\dots,a_n)\\\notag
     & = w(a_1,\dots,\lambda_b(z)a_i,\dots,a_n)\\\label{eqn:redundant}
     & = w(a_1,\dots,a_i\circ \lambda_b(z),\dots,a_n)\\\notag
     & = w(a_1,\dots,\lambda_b(z)\circ a_i,\dots,a_n)
\end{align}
for all $w\in W$, $b,a_1,\dots,a_n\in A$, and $1\leq i\leq n$. Note that \eqref{eqn:redundant} follows from \eqref{eqn:w1} and so is in fact redundant. It is not difficult to check that $M_W(A)$ is indeed a left ideal of $A$. The only thing that might not be obvious is that 
\[ w(a_1,a_2,\dots,a_n) = w(a_1,\dots,\lambda_b(z^{-1})\circ a_i,\dots,a_n)\]
whenever $z\in M_W(A)$. To see why this equality holds, put $c= \lambda_b(z^{-1})$ and\newpage\noindent note that $\overline{c} = \lambda_{\overline{c}}(c^{-1}) = \lambda_{\overline{c}\circ b}(z)$. Since $z\in M_W(A)$, by viewing $\lambda_b(z^{-1})\circ a_i$ as a single element, we obtain 
\begin{align*}
w(a_1,\dots,\lambda_b(z^{-1})\circ a_i,\dots,a_n)
& = w(a_1,\dots,\lambda_{\overline{c}\circ b}(z) \circ \lambda_b(z^{-1})\circ a_i,\dots,a_n)\\
& = w(a_1,\dots,\overline{c}\circ c \circ a_i,\dots,a_n)\\
& = w(a_1,\dots,a_i,\dots,a_n),
\end{align*}
which is as desired.
\end{definition}

Analogous to Proposition \ref{prop:marginal} or \ref{prop:marginal'}, the annihilator of a skew brace may be realized as a marginal left ideal.

\begin{prop}\label{prop:annihilator} Let $A$ be a skew brace. We have
\begin{align*}
    \Ann(A) &=M_{\{[x_1,x_2],x_1*x_2\}}(A)\\&= M_{\{[x_1,x_2],[x_1,x_2]_\circ ,x_1*x_2\}}(A).
\end{align*}
\end{prop}
\begin{proof} Put $W = \{[x_1,x_2],x_1*x_2\}$ and $W' = W\cup \{[x_1,x_2]_\circ\}$. Since
\[M_W(A) \supseteq  M_{W'}(A)\]
clearly holds, it suffices to show that
\[ M_{W'}(A) \supseteq \Ann(A) \supseteq M_W(A).\]
We prove these two inclusions separately.
\begin{enumerate}[(1)]
\item \underline{first inclusion:} Let $z\in \Ann(A)$. Since $\Ann(A)$ is an ideal of $A$, for each word $w\in W'$ and $1\leq i \leq 2$, we only need to check either \eqref{eqn:w1} or \eqref{eqn:redundant}, and we also do not need to consider the action of $\lambda$. This means that it suffices to show that
\begin{align*}
    [a_1,a_2] &= [a_1z,a_2] = [a_1,a_2z]\\
    [a_1,a_2]_\circ &= [a_1\circ z,a_2]_\circ = [a_1,a_2\circ z]_\circ\\
    a_1*a_2 & = (a_1\circ z) * a_2 = a_1*(a_2z)
\end{align*}
for all $a_1,a_2\in A$. The first two sets of equalities hold because $\Ann(A)$ lies in both $Z(A,\cdot)$ and $Z(A,\circ)$. The last set of equalities follow immediately from \eqref{eqn:star3} and \eqref{eqn:star1}. Thus, we indeed have $z\in M_{W'}(A)$.
\item \underline{second inclusion:} Let $z\in M_W(A)$. This implies that
\begin{align*}
    1 = [a,1]=[a,z],\quad
    1 = a*1 = a*z,\quad
    1 = 1*a = z*a,
\end{align*}
for all $a\in A$. Thus, we have  $z\in \Ann(A)$.
\end{enumerate}
This completes the proof.
\end{proof} 

However, we cannot extend Proposition \ref{prop:annihilator} to an arbitrary $n$. For example, for $n=2$, there are skew braces $A$ for which
\[ (A*A)*\Ann_2(A) \neq 1\]
by \cite[Section 3]{Tsang Grun}. This means that the $(n+1)$-fold left-normed star product is not well-defined modulo $\Ann_n(A)$ in general for $n\geq 2$. Nevertheless, the argument in Proposition \ref{prop:marginal} may be adapted to show that the $(n+1)$-fold right-normed star product is well-defined modulo $\Ann_n(A)$ for every $n\geq 1$. This shows that unlike the commutators, the left-normed and right-normed star products exhibit very different behaviors.

\begin{prop}\label{prop:annihilator'} Let $A$ be a skew brace. For any $n\geq 1$, we have
    \[\Ann_n(A) \subseteq M_{\{[x_1,x_2,\dots,x_{n+1}],[x_1,x_2,\dots,x_{n+1}]_\circ,x_1*x_2*\cdots *x_{n+1}\}}(A).\]
\end{prop}
\begin{proof} 
The case $n=1$ follows from Proposition \ref{prop:annihilator}. Now, suppose that the claim holds up to $n-1$. Let $z\in \Ann_n(A)$. Since $\Ann_n(A)$ is an ideal of $A$, for each word and $1\leq i\leq n+1$, we only need to check either \eqref{eqn:w1} or \eqref{eqn:redundant}, and we also do not need to consider the action of $\lambda$. Note that
\begin{align*}
[a_1,a_2,\dots,a_{n+1}]& =[a_1,\dots,a_iz,\dots,a_{n+1}] \\
[a_1,a_2,\dots,a_{n+1}]_\circ & =[a_1,\dots,a_i\circ z,\dots,a_{n+1}]_\circ 
\end{align*}
for all $a_1,a_2,\dots,a_{n+1}\in A$ and $1\leq i\leq n+1$ by Proposition \ref{def:marginal} because 
\[ \Ann_n(A) \subseteq \zeta_n(A,\cdot)\cap \zeta_n(A,\circ)\]
by a simple induction. Thus, it remains to verify that
\begin{align*}
a_1 * \cdots * a_{n+1} &= (a_1\circ z) * a_2\cdots * a_{n+1}\\
& = a_1 * (a_2\circ z) *\cdots *a_{n+1}\\[-2pt]
& \hspace{2cm}\vdots\\
&= a_1 * \cdots * (a_n\circ z) * a_{n+1}\\
& = a_1 * \cdots *a_{n}* (a_{n+1}z)
\end{align*} 
for all $a_1,a_2,\dots,a_{n+1}\in A$. In other words, we check \eqref{eqn:w1} for $1\leq i\leq n$ and \eqref{eqn:redundant} for $i=n+1$. As in the proof of Proposition \ref{prop:annihilator}, the difference in the choice between \eqref{eqn:w1} and \eqref{eqn:redundant} is based on \eqref{eqn:star1} and \eqref{eqn:star3}.
\begin{enumerate}[(a)]
\item \underline{The case $i=n+1$:} Using \eqref{eqn:star1}, we can write
\begin{align*}
a_1 *a_2*\cdots *a_{n} *(a_{n+1}z)= a_1 * \cdots * a_{n-1} * \{ (a_{n} * a_{n+1}) z'\},
\end{align*}
where $z'= a_{n+1} (a_{n} * z ) a_{n+1}^{-1}$. Since $z\in \Ann_{n}(A)$ and $\Ann_{n-1}(A)$ is an ideal of $A$, we deduce that
\[a_n*z\in \Ann_{n-1}(A)\quad\mbox{and}\quad z' \in \Ann_{n-1}(A).\]
Since  $\Ann_{n-1}(A)\subseteq M_{x_1*x_2*\cdots * x_n}(A)$ by induction, it follows that
\[ a_1* \cdots *a_{n} * (a_{n+1}z) = a_1 * \cdots *a_{n-1} * (a_{n} * a_{n+1}), \]
which is as desired.
\smallskip
\item \underline{The case $1\leq i\leq n$:} Put $a = a_{i+1} * \cdots 
* a_{n+1}$. Since the star product is right-normed here, using \eqref{eqn:star3}, we can write
\begin{align*}
    &a_1*\cdots *(a_{i} \circ z)*\cdots *a_{n+1} \\
    &\hspace{1em}= a_1*\cdots *a_{i-1}* \{(a_i\circ z)*a \}\\
   &\hspace{1em} =a_1 * \cdots * a_{i-1} * \{ (a_i * (z * a))(z * a) (a_i *a) \}.
\end{align*}
Analogous to \eqref{eqn:inclusion}, it was shown in \cite[Theorem 6.3]{Tsang survey} that
\[ \Ann_n(A) * A^{n-i+1} \subseteq \Ann_{i-1}(A).\]
Since $z\in \Ann_n(A)$ and $a\in A^{n-i+1}$, we see that
\[ z * a , \, a_i * (z * a)\in \Ann_{i-1}(A).\]
But $\Ann_{i-1}(A)\subseteq M_{x_1*x_2*\cdots *x_i}(A)$ by induction, so it follows that
\[ a_1*\cdots * (a_i\circ z) * \cdots * a_{n+1} = a_1 * \cdots * a_{i-1}* (a_i*a),\]
which is as desired.
\end{enumerate}
This completes the proof.
\end{proof}

\section{Definition of $n$-isoclinism for skew braces}\label{last section}

As mentioned in the introduction, since the $(n+1)$-fold left-normed star product is not well-defined modulo the $n$th annihilator in general for $n\geq 2$, perhaps one should approach $n$-isoclinism for skew braces using verbal sub-skew braces and marginal left ideals. 

\smallskip

In the case of groups, the marginal subgroup of $[x_1,\dots,x_{n+1}]$ happens to be always normal, so one can form the quotient group. But this is not true for an arbitrary collection of words. In the case of skew braces, similarly the marginal left ideal need not be an ideal in general. To resolve this problem, we introduce a new definition that is analogous to the core of a subgroup in a group, so that we can form skew brace quotient.

\begin{lem}\label{lem:core}Let $A$ be a skew brace.
\begin{enumerate}[$(a)$]
\item For any family $\{I_u\}_{u\in U}$ of ideals of $A$ that is totally ordered with respect to inclusion, the union $\bigcup_{u\in U} I_u$ is an ideal of $A$.
\item For any ideals $I$ and $J$ of $A$, the product $IJ$ is an ideal of $A$. 
\item For any sub-skew brace $Z$ of $A$, the set
\[ \mathcal{S}(Z):= \{ I : I \mbox{ is an ideal of $A$ contained in $Z$}\}\]
has a unique maximal element with respect to inclusion.
\end{enumerate}
\end{lem}
\begin{proof}[Proof of $(a)$] 
This is a simple exercise.
\end{proof}
\begin{proof}[Proof of $(b)$] It is clear that $IJ$ is a left ideal of $A$ that is normal in $(A,\cdot)$. For any $a\in I$ and $b\in J$, observe that 
\[ ab = a\circ \lambda_{\overline{a}}(b)\quad\mbox{and}\quad a\circ b=a\lambda_{a}(b),\]
where $\lambda_{\overline{a}}(b),\lambda_a(b)\in J$ because $J$ is a left ideal of $A$. This gives $IJ = I\circ J$, so then $IJ$ is also normal in $(A,\circ)$, whence $IJ$ is an ideal in $A$.
\end{proof}
\begin{proof}[Proof of $(c)$] The set $\mathcal{S}(Z)$, which is partially ordered via inclusion, is non-empty because it contains the trivial sub-skew brace $1$. By (a), every totally ordered subset of $\mathcal{S}(Z)$ has an upper bound in $\mathcal{S}(Z)$, and so Zorn's lemma implies that $\mathcal{S}(Z)$ has a maximal element. By (b), this maximal element is necessarily unique. 
\end{proof}

\begin{definition} Let $A$ be a skew brace. For any sub-skew brace $Z$ of $A$, we define the \textit{core} of $Z$ in $A$, denoted by $\mathrm{Core}_A(Z)$, to be the unique maximal element of $\mathcal{S}(Z)$, namely the unique largest ideal of $A$  that lies in $Z$.
\end{definition}

We now propose a new definition.

\begin{definition}\label{def:nisoclinism}Let $W = W_n$ be a collection of skew brace words in $n+1$ variables. For any skew braces $A$ and $B$, a pair
\[ \xi : A/\mathrm{Core}_A(M_W(A))\longrightarrow B/\mathrm{Core}_B(M_W(B)), \quad \theta : V_W(A)\longrightarrow V_W(B)\]
of isomorphisms is said to be a \textit{$W$-isoclinism} if the diagram
\begin{equation}\label{eqn:diagram'}\begin{tikzcd}[column sep = 2.5cm, row sep = 1.5cm]
 (A/\mathrm{Core}_A(M_W(A)))^{\oplus n+1} \arrow{r}{\phi_w(A)} \arrow{d}[left]{\xi^{\oplus n+1}}& V_W(A)\arrow{d}{\theta}\\
(B/\mathrm{Core}_B(M_W(B)))^{\oplus n+1} \arrow{r}[below]{\phi_w(B)} & V_W(B)
\end{tikzcd}\end{equation}
commutes for every $w\in W$, where the horizontal maps are induced by
\[ \phi_w(-) : (x_1,x_2,\dots,x_{n+1}) \mapsto w(x_1,x_2,\dots,x_{n+1}),\]
and are well-defined by the definition of marginal left ideal. We say that $A$ and $B$ are \textit{$W$-isoclinic} if there is a $W$-isoclinism from $A$ to $B$.
\end{definition}

This new perspective of isoclinism has the advantage that the definition works for all skew braces, regardless of the choice of the collection of words. For example, consider the natural choices
\begin{align}\label{eqn:Wn}
W_{n} &= \{ [x_1,x_2,\dots,x_{n+1}],[x_1,x_2,\dots,x_{n+1}]_\circ,x_1*x_2*\cdots *x_{n+1}\}\\\notag
W_{(n)} & = \{[x_1,x_2,\dots,x_{n+1}],[x_1,x_2,\dots,x_{n+1}]_\circ,x_1\star x_2\star \cdots \star x_{n+1}\}
\end{align}
for $n\geq 1$. Clearly $W:=W_1=W_{(1)}$, and it follows from Proposition \ref{prop:annihilator} and \cite[Proposition 3.9]{isoclinism} that $W$-isoclinism coincides with the isoclinism defined in \cite{isoclinism}. Moreover, by the next proposition, the notions of $W_{n}$-isoclinism and $W_{(n)}$-isoclinism both get weaker as $n$ increases, which is a desirable property.

\begin{prop}Let $W_R$ and $W_L$ be two collections of skew brace words in two variables, and put $W = W_R\cup W_L$. For each $w = w_1\in W$, define
\[w_n(x_1,x_2,\dots,x_{n+1}) = \begin{cases}
    w(x_1,w_{n-1}(x_2,\dots,x_{n+1})) &\mbox{for }w\in W_R,\\
    w(w_{n-1}(x_1,\dots,x_{n}),x_{n+1}) &\mbox{for }w\in W_L,
\end{cases} \]
recursively. For each $n\geq 1$, put $W^n = \{w_n:w\in W\}$, which is a collection of skew brace words in $n+1$ variables. Then two $W^n$-isoclinic skew braces are always $W^{n+1}$-isoclinic.
\end{prop}
\begin{proof}Let $A$ and $B$ be $W^n$-isoclinic skew braces and let $(\xi,\theta)$ be any $W^n$-isoclinism from $A$ to $B$. For simplicity, put
\[ M_k(A) = M_{W^k}(A),\quad C_k(A) = \mathrm{Core}_A(M_k(A)),\quad V_{k+1}(A) = V_{W^k}(A),\]
for any $k$, and similarly for $B$. A simple induction shows that
\[ w_{n+1}(x_1,x_2,\dots,x_{n+2}) = \begin{cases}
    w_n(x_1,\dots,x_n,w(x_{n+1},x_{n+2}))&\mbox{for }w\in W_R,\\
    w_n(w(x_1,x_2),x_3,\dots,x_{n+2})&\mbox{for }w\in W_L.
\end{cases}\]
Together with the original definition of $w_{n+1}$, we easily see that
\[ C_{n}(A) \subseteq C_{n+1}(A),\quad 
V_{n+1}(A)\supseteq V_{n+2}(A),\]
and similarly for $B$. In what follows, for each $a\in A$, let $\widetilde{a}$ denote the class of $a$ modulo $C_{n}(A)$, and similarly for $B$. Also, let $w_{n+1}\in W^{n+1}$ be arbitrary. Here, let us assume that $w\in W_R$. We shall omit the case $w\in W_L$ because it may be dealt with in a completely analogous manner.

\smallskip

First, let us show that
\[\xi(C_{n+1}(A)/C_n(A)) = C_{n+1}(B)/C_{n}(B),\]
which would imply that $\xi$ induces an isomorphism
\[ \xi' : A/C_{n+1}(A) \longrightarrow B/C_{n+1}(B).\]
By symmetry, it suffices to show the left-to-right inclusion. To that end, let $I_{n+1}(B)$ be the ideal of $B$ such that
\[ \xi(C_{n+1}(A)/C_n(A)) = I_{n+1}(B)/C_n(B),\]
and we wish to prove $I_{n+1}(B)\subseteq C_{n+1}(B)$. Since $I_{n+1}(B)$ is an ideal of $B$, by the maximality of $C_{n+1}(B)$ it suffices to prove $I_{n+1}(B)\subseteq M_{n+1}(B)$, and for each $z\in I_{n+1}(B)$, we only need to check that
\[ w_{n+1}(b_1,b_2,\dots,b_{n+2}) = w_{n+1}(b_1,\dots,b_iz,\dots,b_{n+2}) \]
for all $b_1,b_2,\dots,b_{n+2}\in B$ and $1\leq i\leq n+2$. Now, we may write
\[ \xi(\widetilde{u}) = \widetilde{z},\quad\xi(\widetilde{a_1})=\widetilde{b_1},\,\ \dots\,\ ,\quad \xi(\widetilde{a_{n+2}})=\widetilde{b_{n+2}}\]
for some $u\in C_{n+1}(A)$ and $a_1,\dots,a_{n+2}\in A$. By the commutativity of \eqref{eqn:diagram'} for the word $w_n$ and the fact that $u\in C_{n+1}(A)$, we then see that
\begin{align*}
    w_{n+1}(b_1,b_2,\dots,b_{n+2}) &  = w_n(b_1,\dots,b_n,w(b_{n+1},b_{n+2}))\\
    &= \theta(w_n(a_1,\dots,a_n,w(a_{n+1},a_{n+2})))\\
    & = \theta(w_{n+1}(a_1,a_2,\dots,a_{n+2}))\\
    & = \theta(w_{n+1}(a_1,\dots,a_iu,\dots,a_{n+2}))\\
    & = \theta(w_n(a_1,\dots,a_iu,\dots,a_n,w(a_{n+1},a_{n+2}))) \tag{$\dagger$}\\
    & = w_n(b_1,\dots,b_iz,\dots,b_n,w(b_{n+1},b_{n+2}))\tag{$\dagger$}\\
    & = w_{n+1}(b_1,\dots,b_iz,\dots,b_{n+2}).
\end{align*}
In the equalities ($\dagger$), we assumed $1\leq i\leq n$, but the same calculation works for $n+1\leq i\leq n+2$ as well. Since $w_{n+1}\in W^{n+1}$ was arbitrary, it follows that $z\in M_{n+1}(B)$, as desired.

\smallskip

Next, let us show that 
\[\theta(V_{n+2}(A)) = V_{n+2}(B),\]
which would imply that $\theta$ induces an isomorphism
\[ \theta': V_{n+2}(A)\longrightarrow V_{n+2}(B).\]
By symmetry, it suffices to show the left-to-right inclusion. By the commutativity of \eqref{eqn:diagram'} for the word $w_n$, for any $a_1,a_2,\dots,a_{n+2}\in A$, we have
\begin{align*}
    \theta(w_{n+1}(a_1,a_2,\dots,a_{n+2})) & = \theta(w_n(a_1,\dots,a_n,w(a_{n+1},a_{n+2})))\\
    &= w_n(\xi(\widetilde{a_1}),\dots,\xi(\widetilde{a_n}),\xi(w(\widetilde{a_{n+1}},\widetilde{a_{n+2}})))\\
    &= w_n(\xi(\widetilde{a_1}),\dots,\xi(\widetilde{a_n}),w(\xi(\widetilde{a_{n+1}}),\xi(\widetilde{a_{n+2}})))\\
    & =w_{n+1}(\xi(\widetilde{a_1}),\xi(\widetilde{a_2}),\dots,\xi(\widetilde{a_{n+2}}))\in V_{n+2}(B).
\end{align*}
Since the elements $w_{n+1}(a_1,a_2,\dots,a_{n+2})$ generate $V_{n+2}(A)$ as $w_{n+1}$ ranges over $W^{n+1}$, it follows that $\theta(V_{n+2}(A))\subseteq V_{n+2}(B)$, as desired.

\smallskip

The pair $(\xi',\theta')$ is clearly a $W^{n+1}$-isoclinism. Thus, we have shown that $A$ and $B$ are $W^{n+1}$-isoclinic, as claimed.
\end{proof}

Finally, let us go back to the collection $W_n$ of skew brace words as stated in \eqref{eqn:Wn}. For any skew brace $A$, we know from Proposition \ref{prop:annihilator'} that 
\[\Ann_n(A)\subseteq \mathrm{Core}_A(M_{W_n}(A)) \subseteq  M_{W_n}(A),\]
for any $n\geq 1$. In this case, instead of
\[ A/\mathrm{Core}_A(M_{W_n}(A)),\]
we can also define $W_n$-isoclinism via the possibly more natural choice
\[ A/\Ann_n(A).\]
At the moment, it is unclear, at least to us, which is the better choice. We end this paper with a few examples of $A$ such that
\begin{equation}\label{eqn:condition} \Ann_2(A) \subsetneq \mathrm{Core}_A(M_{W_2}(A)),\end{equation}
which means that the two choices are genuinely different. To that end, it is enough to exhibit a skew brace $A$ satisfying
\begin{equation}\label{eqn:ex}\gamma_3(A,\cdot) = \gamma_3(A,\circ) = A^3=1\quad\mbox{and}\quad \Ann_2(A)\neq A.\end{equation}
Indeed, the equalities trivially imply that
\[ \mathrm{Core}_A(M_{W_2}(A)) = M_{W_2}(A) = A,\]
and so we obtain the desired condition \eqref{eqn:condition}. We shall exhibit examples by considering braces (i.e. skew braces $A$ with $(A,\cdot)$ abelian) of order $p^3$ for a prime $p$. In this case, we have
\[ \gamma_2(A,\cdot) = \gamma_3(A,\cdot) =1\mbox{\quad and \quad}\gamma_3(A,\circ) = 1\]
because groups of order $p^3$ have nilpotency class at most $2$. Hence, we only need to check that $A^3 = 1$ and $\Ann_2(A)\neq A$.

\smallskip

Braces of order $p^3$ were classified in \cite{bachiller}. The following examples are:
\begin{enumerate}[$\bullet$]
\item the last brace in \cite[Theorem 3.1(2), Socle of order $2$]{bachiller},
\item the last two braces in \cite[Theorem 3.2(2), Socle of order $p$]{bachiller} with $a=0$,
\item the last brace in \cite[Theorem 3.2(3), Socle of order $p$]{bachiller} with $c=0$,
\end{enumerate}
respectively. We show that they all satisfy $A^3 =1$ and $\Ann_2(A)\neq A$.

\smallskip

In what follows, we use the notation
\[ C(a,2)= 1 + 2 + \cdots + (a-1) = \frac{a(a-1)}{2}\]
for any non-negative integer $a$.

\begin{example} Consider the brace 
\[ A = (\mathbb{Z}/2\mathbb{Z}\times\mathbb{Z}/4\mathbb{Z},+,\circ),\]
where $+$ is the usual addition, and $\circ$ is defined by
\[ \begin{pmatrix}
    a_1 \\ a_2\end{pmatrix}\circ \begin{pmatrix}b_1\\b_2\end{pmatrix} = \begin{pmatrix} a_1 + b_1 + C(a_2,2)b_2 \\ a_2 + b_2 + 2a_2b_2\end{pmatrix}.\]
The star product is then given by
\[ \begin{pmatrix}
    a_1 \\ a_2\end{pmatrix}* \begin{pmatrix}b_1\\b_2\end{pmatrix} = \begin{pmatrix} C(a_2,2)b_2 \\  2a_2b_2\end{pmatrix}.\]
From this, it is easy to see that
\[ A^2 = \mathbb{Z}/2\mathbb{Z} \times 2\mathbb{Z}/4\mathbb{Z},\quad A^3= \{0\}\times \{0\},\]
\[ \Ann(A) = \mathbb{Z}/2\mathbb{Z}\times \{0\},\quad \Ann_2(A)=\mathbb{Z}/2\mathbb{Z}\times 2\mathbb{Z}/4\mathbb{Z}.\]
Thus, this brace $A$ satisfies the condition \eqref{eqn:ex}.
\end{example}

\begin{example} Let $p$ be any odd prime. Consider the brace 
\[ A = (\mathbb{Z}/p\mathbb{Z}\times\mathbb{Z}/p^2\mathbb{Z},+,\circ),\]
where $+$ is the usual addition, and $\circ$ is defined by
\[ \begin{pmatrix}
    a_1 \\ a_2\end{pmatrix}\circ \begin{pmatrix}b_1\\b_2\end{pmatrix} = \begin{pmatrix} a_1 + b_1 + a_2b_2\\ a_2 + b_2 + p(a_1b_2 - C(a_2,2)b_2)\zeta\end{pmatrix}.\]
Here, either $\zeta=1$ or $\zeta=\varepsilon$ (following the notation of \cite{bachiller}) is a non-quadratic residue modulo $p$. The star product is then given by
\[ \begin{pmatrix}
    a_1 \\ a_2\end{pmatrix}* \begin{pmatrix}b_1\\b_2\end{pmatrix} = \begin{pmatrix}  a_2b_2 \\ p(a_1b_2 - C(a_2,2)b_2)\zeta\end{pmatrix}.\]
From this, it is easy to see that
\[ A^2 = \mathbb{Z}/p\mathbb{Z} \times p\mathbb{Z}/p^2\mathbb{Z},\quad A^3= \{0\}\times \{0\},\]
\[ \Ann(A) = \{0\}\times p\mathbb{Z}/p^2\mathbb{Z},\quad \Ann_2(A)=\mathbb{Z}/p\mathbb{Z}\times p\mathbb{Z}/p^2\mathbb{Z}.\]
Thus, this brace $A$ satisfies the condition \eqref{eqn:ex}.
\end{example}

\begin{example} Let $p$ be any odd prime. Consider the brace 
\[ A = (\mathbb{Z}/p\mathbb{Z}\times\mathbb{Z}/p\mathbb{Z}\times \mathbb{Z}/p\mathbb{Z},+,\circ),\]
where $+$ is the usual addition, and $\circ$ is defined by
\[ \begin{pmatrix}
    a_1 \\ a_2 \\a_3\end{pmatrix}\circ \begin{pmatrix}b_1\\b_2\\b_3\end{pmatrix} = \begin{pmatrix} a_1 + b_1 + (a_2 -C(a_3,2))b_2 \\ a_2 + b_2 + a_3b_3\\a_3+b_3\end{pmatrix}.\]
The star product is then given by
\[ \begin{pmatrix}
    a_1 \\ a_2 \\a_3\end{pmatrix}*\begin{pmatrix}b_1\\b_2\\b_3\end{pmatrix} = \begin{pmatrix} (a_2 -C(a_3,2))b_3 \\  a_3b_3\\0\end{pmatrix}.\]
From this, it is easy to see that
\[ A^2 = \mathbb{Z}/p\mathbb{Z}\times\mathbb{Z}/p\mathbb{Z}\times\{0\},\quad A^3= \{0\}\times \{0\}\times\{0\},\]
\[ \Ann(A) = \mathbb{Z}/p\mathbb{Z}\times\{0\}\times\{0\},\quad \Ann_2(A)= \mathbb{Z}/p\mathbb{Z}\times\mathbb{Z}/p\mathbb{Z}\times\{0\}.\]
Thus, this brace $A$ satisfies the condition \eqref{eqn:ex}.
\end{example}

\section*{Acknowledgments}

This research is supported by JSPS KAKENHI Grant Number 24K16891.

\end{document}